\documentclass[10pt]{amsart}
\usepackage{amsmath,amssymb,amscd,amsthm,latexsym,graphics,enumerate,engord}
\usepackage[all]{xy}
\usepackage{pstricks-add}

\newtheorem{thm}{Theorem}[section]

\newtheorem{prop}[thm]{Proposition}
\newtheorem{cor}[thm]{Corollary}
\newtheorem{defin}[thm]{Definition}
\newtheorem{rem}[thm]{Remark}

\newdir{ >}{{}*!/-7pt/\dir{>}}

\def\R{{\mathbb{R}}}

\def\disfrac#1#2{{\displaystyle{\frac{#1}{#2}  }}}

\def\cat{{\rm{cat}\hskip1pt}}

\def\secat{{\rm{secat}\hskip1pt}}

\def\nil{{\rm{nil}\hskip1pt}}
\def\TC{{\rm{TC}\hskip1pt}}

\title{Q-Topological Complexity}

\author[L. Fern\'andez Su\'arez, L. Vandembroucq]{Luc\'{\i}a Fern\'andez Su\'arez and Lucile Vandembroucq}
\thanks{This research has been supported by Portuguese Funds from the ``Funda\c c\~ao para a Ci\^encia e a Tecnologia'', through the Project UID/MAT/0013/2013.}

\begin{document}

\begin{abstract}
By analogy with the invariant Q-category defined by Scheerer, Stanley and Tanr\'e, we introduce the notions of Q-sectional category and Q-topological complexity. We establish several properties of these invariants. We also obtain a formula for the behaviour of the sectional category with respect to a fibration which generalizes the classical formulas for Lusternik-Schnirelmann category and topological complexity. 
\end{abstract}

\maketitle


\section{Introduction}

 The $Q$-\textit{category} of a topological space $X$, denoted by $Q\cat X$, is a lower bound for the Lusternik-Schnirelmann category of $X$, $\cat X$, which has been introduced by H. Scheerer, D. Stanley, and D. Tanr\'e in \cite{SST}. This invariant, defined using a fibrewise extension of a functor $Q^k$ equivalent to $\Omega^k\Sigma^k$,  has been in particular used in the study of critical points (see \cite[Chap. 7]{CLOT}, \cite{MoyauxV}) and in the study of the Ganea conjecture. More precisely, after N. Iwase \cite{Iwase} showed that the Ganea conjecture, which asserted that $\cat (X\times S^n) =\cat X +1$ for any $n\geq 1$, was not true in general, although it was known to be true for many classes of spaces (e.g. \cite{Singhof}, \cite{Jessup}, \cite{Hess}, \cite{Rudyak}, \cite{Strom}), one could ask for a complete characterization of the spaces $X$ satisfying the equality above. In \cite{SST}, H.  Scheerer, D.  Stanley, and D.  Tanr\'e conjectured that a finite  CW-complex $X$ satisfies the Ganea conjecture, that is the equality $\cat (X\times S^n)
=\cat X +1$ holds for any $n\geq 1$, if and only if $Q\cat X =\cat X$. One direction of this equivalence has been proved in \cite{Qcat} but the complete answer is still unknown.

In this paper we introduce the analogue of $Q$-category for Farber's topological complexity \cite{Farber} and establish some properties of this invariant. Since L.-S. category and topological complexity are both special cases
of the notion of \emph{sectional category}, introduced by A. Schwarz in \cite{Schwarz}, we naturally consider and study a notion of $Q$-sectional category. Our definition, given in Section \ref{DefQsecat}, is based on a generalized notion of Ganea fibrations and on a fibrewise extension of the functor $Q^k$ that we respectively recall in Sections \ref{Ganeafibrations} and \ref{Qconstruction}. We next, in Section \ref{properties}, establish various formulas for $Q$-sectional category and $Q$-topological complexity. 
It is worth noting that our study of the behaviour of the $Q$-sectional category in a fibration led us to establish a new formula for the sectional category (see Theorem \ref{secatfibration}), which generalizes both the classical formula for the LS-category in a fibration and the formula established by M. Farber and M. Grant for topological complexity \cite{FarberGrant}.
Finally, we use our results to study some examples. This permits us in particular to observe that the analogue of the Scheerer-Stanley-Tanr\'e conjecture for topological complexity is not true. We also include a small observation about the original Scheerer-Stanley-Tanr\'e conjecture.

Throughout this text we work in the category of compactly generated Hausdorff spaces having the homotopy type of a CW-complex.
\section{Q-sectional category and Q-topological complexity}

\subsection{Sectional category and Ganea fibrations}\label{Ganeafibrations}
Let $f: A\to X$ be a map where $X$ is a well-pointed path-connected space with base point $\ast \in X$. The \emph{sectional category} of $f$,
$\secat(f)$, is the least integer $n$ (or $\infty$) for which there exists an open cover 
$U_0,\ldots, U_n$ of $X$ such that, for any $0\leq i\leq n$, $f$ admits a local homotopy section on $U_i$ (that is, a continuous map $s_i:U_i\to A$ such that $f\circ s_i$ is homotopic to the inclusion $U_i\hookrightarrow X$).
When $f$ is a fibration, we can, equivalently, require local strict sections instead of local homotopy sections. As special cases of sectional category, Lusternik-Schnirelmann category and Farber's topological complexity are respectively given by
\begin{itemize}
\item $ \cat(X)=\secat(ev_1:PX\to X)=\secat(\ast \to X)$ 

\noindent where $PX\subset X^I$ is the space of paths beginning at the base point $*$ and $ev_1$ is the evaluation map at the end of the path,
\item $\TC(X)=\secat(ev_{0,1}:X^I\rightarrow X\times X) =\secat(\Delta:X\to X\times X)$

\noindent where $ev_{0,1}$ evaluates a path at its extremities and $\Delta$ is the diagonal map.
\end{itemize} 
As is well-known, if $f:A\to X$ and $g:B\to Y$ are two maps with homotopy equivalences $A\stackrel{\sim}{\to} B$ and $X\stackrel{\sim}{\to} Y$ making the obvious diagram commutative then $\secat(f)=\secat(g)$. Also $\secat(f)$ can be characterized through the existence of a global section for a certain join map which can be, for instance, explicitly constructed through an iterated fibrewise join of $f$ when $f$ is a fibration (\cite{Schwarz}). Here we will assume (without loss of generality) that $f:A\to X$ is a (closed) pointed cofibration and consider the following constructions which give us a natural and explicit fibration, the Ganea fibration of $f$, equivalent to the join map characterizing $\secat(f)$:

\begin{itemize}
\item the fatwedge of $f$ (\cite{Fasso}, \cite{weaksecat}):
$$T^ n(f)=\{(x_0,\dots,x_n)\in X^{n+1} ~|~  \exists j, x_j\in f(A)\}$$
\noindent which generalizes the classical fat-wedge $$T^ n(X)=\{(x_0,\dots,x_n)\in X^{n+1} ~|~ \exists j, x_j= *\},$$
\item the space $\Gamma_nX=\left\{(\gamma_0,\dots,\gamma_n)\in \left(
X^{I}\right)^{n+1}~|~\gamma_0(0)=\dots=\gamma_n(0)\right\}$\\

\noindent together with the fibration $$\Gamma_nX\to X, \quad (\gamma_0,\dots,\gamma_n)\mapsto \gamma_0(0)=\dots=\gamma_n(0)$$ which is a homotopy equivalence,\\

\item the $n$th Ganea fibration of $f$, $g_n(f):G_n(f)\to X$, which is obtained by pull-back along the diagonal map $\Delta_{n+1}:X\to X^{n+1}$ of the fibration associated to the inclusion $T^n(f)\hookrightarrow X^{n+1}$ and explicitly given by:
$$\begin{array}{rcl}
G_n(f)=\left\{(\gamma_0,\dots,\gamma_n)\in \Gamma_n(X)~|~ \exists j,~\gamma_j(1)\in f(A)\right\}&\to& X\\
(\gamma_0,\dots,\gamma_n)&\mapsto &\gamma_0(0)=\dots=\gamma_n(0).
\end{array}$$ 
\end{itemize}

\noindent All these spaces are considered with the obvious base points $(\ast,\ldots,\ast)\in T^n(f)$ and $(\hat{\ast},\ldots,\hat{\ast})\in G_n(f)\subset \Gamma_n(X)$ (where $\hat{\ast}$ denotes the constant path). By construction, there exists a commutative diagram in which the square is a (homotopy) pull-back.
$$\xymatrix{
G_n(f)\ar[r]\ar[d]_{g_n(f)} & \bullet \ar@{->>}[d] & \ar[l]_{\sim} T^ n(f)\ar@{^(->}[dl] \\
X \ar[r]_{\Delta_{n+1}}     & X^ {n+1}            }$$
By \cite{Fasso}, we know that $\secat(f)$ is the least integer $n$ such that the diagonal map $\Delta_{n+1}:X\to X^{n+1}$ lifts up to homotopy in the fatwedge of $f$, $T^n(f)$. By the homotopy pull-back diagram above, this is equivalent to say that $\secat(f)$ is the least integer $n$ such that the fibration $g_n(f)$ admits a (homotopy) section. By \cite[Th. 8]{weaksecat}, the fibration $g_n(f)$ is equivalent to the (fibrewise) join of $n+1$ copies of $f$ or of any map weakly equivalent to $f$. The fibre of $g_n(f)$, denoted by $F_n(f)$, is homotopically equivalent to the (usual) join of $n+1$ copies of the homotopy fibre $F$ of $f$.

When $f$ is the inclusion $\ast \to X$ (which is a cofibration since $X$ is well-pointed) we recover a possible description of the classical Ganea fibration of $X$ and we will use, in that case, the classical notation
$$g_n(X):G_n(X)=\left\{(\gamma_0,\dots,\gamma_n)\in \Gamma_n(X)~:~ \exists j,~\gamma_j(1)=\ast
\right\}\to X.$$ We have  
$$\cat(X)\leq n \quad \Leftrightarrow \quad g_n(X):G_n(X)\to X \mbox{ admits a (homotopy) section}.$$ 
When $X$ is a CW-complex, $f=\Delta:X\to X\times X$ is a cofibration and we have
$$\TC(X)\leq n \quad \Leftrightarrow \quad g_n(\Delta):G_n(\Delta)\to X\times X \mbox{ admits a (homotopy) section}.$$

We finally note that, if we have a commutative diagram where $f$ and $g$ are cofibrations
$$\xymatrix{ A \ar[r]^{\varphi}\ar[d]_{f} & B \ar[d]^{g}\\
X \ar[r]_{\psi}& Y
}$$
then we have a commutative diagram
$$\xymatrix{ G_n(f) \ar[r]^{G_n(\psi,\varphi)}\ar[d]_{g_n(f)} & G_n(g) \ar[d]^{g_n(g)}\\
X \ar[r]_{\psi}& Y.
}$$
Moreover, we have
\begin{enumerate}
\item[(a)] if $\varphi$ and $\psi$ are homotopy equivalences, then so is $G_n(\psi,\varphi)$;
\item[(b)] if the first diagram is a homotopy pull-back, then so is the second diagram.
\end{enumerate}

\noindent Note that $(b)$ follows (for instance) from the equivalence between $g_n(f)$ and the fibrewise join of $n+1$ copies of $f$ together with the Join Theorem (\cite{Doeraene}).

\subsection{Fibrewise $Q$ construction}\label{Qconstruction}
The notion of $Q$-category defined in \cite{SST} is based on the Dror-Frajoun fibrewise extension \cite{Farjoun} of a functor $Q^k$ equivalent to $\Omega^ k\Sigma^ k$. Instead of using the Dror-Farjoun construction, we will here use the (equivalent) explicit fibrewise extension of $Q^k$ given in \cite{Qcat}. We describe the case $k=1$, recall more quickly the general case and refer to \cite{SST},\cite[Sections 4.5, 4.6, 4.7]{CLOT} and \cite{Qcat} for further details.

We denote by $\Sigma^ 1 Z$ the unreduced suspension of a space $Z$, \textit{i.e.} $\Sigma^ 1 Z = Z\times I / \sim$ with $(z,0)\sim (z',0)$ and $(z,1)\sim
(z',1)$ for $z,z' \in Z$. Denoting by $[z,t]\in \Sigma^1Z$ the class of $(z,t)\in Z\times I$, we have a canonical map $\alpha^1=\alpha^1_Z:\{0,1\}=\partial I \to \Sigma^ 1 Z$ given by $\alpha^1(0)=[z,0]$ and $\alpha^1(1)=[z,1]$ where $z\in Z$. The functor $Q^ 1$ is defined by
$$Q^ 1(Z)=\{ \omega : I\to \Sigma^ 1 Z | \omega_{|\partial I}= \alpha^1_Z\}$$
and is given with a co-augmentation $\eta^1=\eta^1_Z:Z\to Q^ 1Z$ which takes $z\in Z$ to the path $z\to [z,t]$. If $Z$ is pointed and $\widetilde{\Sigma}Z$ denotes the reduced suspension of $Z$, then there is a natural map $Q^ 1(Z)\to \Omega \widetilde{\Sigma} Z$, which is a homotopy equivalence induced by the identification map $\Sigma^1Z\stackrel{\sim}{\to} \widetilde{\Sigma}Z$ and which makes compatible the coaugmention of $Q^ 1$ with the usual coaugmentation $\tilde{\eta}:Z\to \Omega \widetilde{\Sigma} Z$.\\

We now describe a fibrewise extension of $Q^1$. Let $p:E\to B$ be a fibration (over a path-connected space) with fibre $F$. The fibrewise suspension of $p$,
$\Sigma^1_BE\to B$, is defined by the push-out:
$$\xymatrix{
E\times\{0,1\} \ar@{ >->}[r] \ar[d]_{p\times id} & E \times I \ar[d] \ar@/^/[rdd]&\\
B\times\{0,1\} \ar@{ >->}[r] \ar@/_/[drr]& \Sigma^1_BE \ar[dr]^{\hat p}&\\
&&B
}$$
The resulting map $\hat{p}:\Sigma^1_BE\to B$ is a fibration whose fibre over $b$ is the (unreduced) suspension $\Sigma^1 F_b$ of the fibre of $p:E\to B$ over $b$. By construction, we have a canonical map:
$$\mu^1:B\times\{0,1\}\to \Sigma^1_BE$$
which is a fibrewise extension of $\alpha^1$: for any $b\in B$, $\mu^1(b,-)=\alpha^1_{F_b}:\{0,1\}\to \Sigma^1 F_b$. We define
$$Q^1_B (E) =\{ \omega : I\to \Sigma^1_B E ~|~ \exists b\in B , \hat p \omega = b \mbox{  and  } \omega_{|\partial I}= \mu^1 (b,-)\}$$
together with the map  $q^1(p):Q^1_B (E)\to B$  given by $\omega\mapsto \hat p \omega(0)$. This is a fibration whose fibre over $b$ is 
$Q^1 (F_b)$ (\cite[Lemma 8]{Qcat}). We also have a fibrewise coaugmentation $\eta^1_B: E\to Q^1_B(E)$ which extends $\eta^1: F\to Q^1(F)$ and we have a commutative diagram

$$\xymatrix{
E \ar[d]^p \ar[r]^{\eta^1_B} & Q^1_B(E) \ar[d]^{q^1(p)}\ar[r] & Q^1(E)\ar[d]^{Q^1(p)}\ar[r]^{\sim} & \Omega\widetilde{\Sigma} E\ar[d]^{\Omega\widetilde{\Sigma} p} \\
B \ar@{=}[r] & B \ar[r]_{\eta^1}& Q^1(B)\ar[r]^{\sim} & \Omega\widetilde{\Sigma} B
}$$
where the map $Q^1_B(E)\to Q^1(E)$ is induced by the identification map $\Sigma^1_BE\to \Sigma^1E$.\\

\noindent In general, for any $k\geq 1$, we consider
\begin{itemize}
\item the $k$ fold unreduced suspension of a space $Z$, $\Sigma^kZ$, which can be described as $(Z\times I^k)/\sim$ where the relation is given by
$$(z,t_1,\ldots,t_k)\sim (z',t'_1,\ldots,t'_k) $$
if, for some $i$, $t_i=t'_i\in \{0,1\}$ and $t_j=t'_j$ for all $j>i$.
We write $[z,t_1,\ldots,t_k]$ for the class of an element. 
\item the $k$th fibrewise suspension of $p: E\to B$, $\hat p^k:\Sigma_B^kE \to B$,
whose fibre over $b$ is $\Sigma^kF_b$. As before, $\Sigma_B^kE$ can be described as $(E\times I^k)/\sim$ where the relation is given by
$(e,t_1,\ldots,t_k)\sim (e',t'_1,\ldots,t'_k) $
if $p(e)=p(e')$ and, for some $i$, $t_i=t'_i\in \{0,1\}$ and $t_j=t'_j$ for all $j>i$. We write $[e,t_1,\ldots,t_k]$ for the class of an element.
\item the canonical map $\alpha^k:\partial I^k \to \Sigma^k Z$ given by $\alpha^k(t_1,\ldots,t_k)=[z,t_1,\ldots,t_k]$ (where $z\in Z$ is any element) and its fibrewise extension $\mu^k:B\times \partial I^k \to \Sigma_B^k E$ 
satisfying $\mu^ k(b,-)=\alpha^k:\partial I^k \to \Sigma^k F_b$ for any $b\in B$.
\item the fibration $q^k(p):Q_B^k(E) \to B$ where $Q_B^kE$ is the (closed) subspace of $(\Sigma_B^k E)^{I^k}$ given by
$$\qquad Q^k_B (E) =\{ \omega : I^k\to \Sigma^k_B E ~|~ \exists b\in B , \hat p^k \omega = b \mbox{  and  } \omega_{|\partial I^k}= \mu^k (b,-)\}$$
and  $q^k(p)(\omega)=\hat{p}^k\omega(0)$. The fibre is
$$Q^k(F)=\{\omega : I^k\to \Sigma^k F | \omega_{|\partial I^k}=\alpha^k\}\stackrel{\sim}{\to}\Omega^k\widetilde{\Sigma}^k F$$
and is given with an obvious coaugmentation $\eta^k:F\to Q^k(F)$, equivalent to the classical augmentation $\tilde{\eta}^k:F\to  \Omega^k\widetilde{\Sigma}^k F$. We denote by $\eta_B^k:E\to Q_B^k(E)$ the fibrewise extension of $\eta^k$. When it is relevant $Q^k(F)$ and $Q_B^kE$ are considered with the base point given by the map $u\mapsto [*,u]$ where $u\in I^k$ and $\ast$ is the base point of $F$ and $E$.
\end{itemize}
We have, for any $k\geq 1$, a commutative diagram:
\begin{equation}\label{DiagQk}
\xymatrix{
E \ar[d]^p \ar[r]^{\eta^k_B} & Q^k_B(E) \ar[d]^{q^k(p)}\ar[r] & Q^k(E)\ar[d]^{Q^k(p)}\ar[r]^{\sim} & \Omega^k\widetilde{\Sigma}^k E\ar[d]^{\Omega^k\widetilde{\Sigma}^k p} \\ 
B \ar@{=}[r] & B \ar[r]^{\eta^k}& Q^k(B)\ar[r]^{\sim} & \Omega^k\widetilde{\Sigma}^k B.
}
\end{equation}
Setting $Q^0=Q^0_B=id$ we also have a commutative diagram of the following form:
\begin{equation}\label{DiagSequenceQk}\xymatrix{
E=Q^0_BE \ar[d]^p \ar[r] & Q^1_B(E) \ar[d]^{q^1 (p)} \ar[r]& \cdots \ar[r]&Q^k_B(E) \ar[d]^{q^k(p)} \ar[r]^{b^k_B} & Q^{k+1}_B(E) \ar[d]^{q^{k+1} (p)} \ar[r]& \cdots\\
B \ar@{=}[r] & B \ar@{=}[r]& \cdots\ar@{=}[r] &B \ar@{=}[r] & B \ar@{=}[r]& \cdots
}
\end{equation}
where the map $b^k_B$ is a fibrewise extension of the map $b^k:Q^k(F)\to Q^{k+1}(F)$ given by $b^k(\omega)(t_1,\ldots,t_{k+1})=[\omega(t_1,\ldots,t_k),t_{k+1}]$. Observe that the map $b^k$ is equivalent to the map $\Omega^k(\tilde{\eta}_{\widetilde{\Sigma}^kF}):\Omega^k(\widetilde{\Sigma}^kF) \to \Omega^k(\Omega \widetilde{\Sigma} \widetilde{\Sigma}^kF)=\Omega^{k+1}\widetilde{\Sigma}^{k+1}F$. We can interpret the sequence (\ref{DiagSequenceQk}) as a fibrewise stabilization of the fibration $p$ since, for any integers $k$ and $i$, we have $\pi_i(Q^kF)=\pi_{i+k}(\widetilde{\Sigma}^kF)$.\\

\noindent For any $k\geq 0$, the functors $Q^k$ and $Q^k_B$ preserve homotopy equivalences. We also note that the $Q^k_{-}$ construction is natural in the sense that, if we have a commutative diagram
\[\begin{split}\xymatrix{%
E' \ar[r]^f\ar[d]_{p'} & E \ar[d]^{p}\\
B' \ar[r]_{g}& B}\end{split}
\label{diag1}\tag{\dag}\]
where $p$ and $p'$ are fibrations (over path-connected spaces), then, for any $k$, we obtain commutative diagram
\[\begin{split}\xymatrix{%
Q^k_{B'}(E') \ar[r]\ar[d]_{q^k(p')} & Q^k_B(E) \ar[d]^{q^k(p)}\\
B' \ar[r]_{g} & B
}\end{split}
\label{diag2}\tag{\ddag}\]
Moreover we have
\begin{prop}\label{funct+pullback}
With the notations above, if Diagram (\ref{diag1}) is a homotopy pull-back, then so is Diagram (\ref{diag2}).
\end{prop}

\begin{proof} Since the functors $Q^k$ and $Q^k_B$ preserve homotopy equivalences, it is sufficient to establish the statement when Diagram (\ref{diag1}) is a strict pull-back. In this case, the whisker map $\psi: E'\to E\times_{B}B'$ is a homeomorphism and its inverse $\phi$ satisfies $p'\phi(e,b')=b'$ for $(e,b')\in E\times_{B}B'$. We can then check that the map $\bar{\phi}:\Sigma^k_{B}E\times_{B}B'\to \Sigma^k_{B'}E'$ induced by $((e,u),b')\mapsto  [\phi(e,b'),u]$ for $((e,u),b')\in (E\times I^k)\times_BB'$ 
is an inverse of the whisker map $\Sigma^k_{B'}E'\to \Sigma^k_{B}E\times_{B}B'$. In other words, the diagram 
$$\xymatrix{ \Sigma^k_{B'}E' \ar[r]\ar@{->>}[d]_{\widehat{p'}^k} & \Sigma^k_BE \ar@{->>}[d]^{\hat{p}^ k}\\
B' \ar[r]_{g} & B
}$$
is a strict pull-back  (and homotopy pull-back since $\hat{p}^k$ is a fibration). We can next check that the map $Q^k_{B}(E)\times_{B}B'\to Q^k_{B'}(E')$ that takes $(\omega,b')\in Q^k_{B}(E)\times_{B}B'$ to the element $I^k\to \Sigma^k_{B'}E'$ of $Q^k_{B'}(E')$ given by $u\mapsto \bar{\phi}(\omega(u),b')$ is an inverse of the whisker map $Q^k_{B'}(E')\to Q^k_{B}(E)\times_{B}B'$. This means that Diagram (\ref{diag2}) is a strict pull-back and therefore a homotopy pull-back since $q^k(p)$ is a fibration. 
\end{proof}

\begin{rem} Assuming that the fibres $F$ and $F'$ of $p$ and $p'$ are path-connected spaces (having the homotopy type of a CW-complex), we can give the following more conceptual proof of Proposition \ref{funct+pullback}. In this case, the spaces $Q^k_B(E)$ and $Q^k_{B'}(E')$ are path-connected spaces having the homotopy type of a CW-complex, since $B$ and $Q^k (F)\simeq\Omega^k\widetilde{\Sigma}^k F$ and $Q^k (F')\simeq \Omega^k\widetilde{\Sigma}^k F$ so are (see \cite{Stasheff}). Since Diagram (\ref{diag1}) is a homotopy pull-back, the map $e:F'\to F$ induced by this diagram is a homotopy equivalence and so is $Q^k(e)$. Therefore Diagram (\ref{diag2})
is a homotopy pull-back since the whisker map from $Q^k_{B'}(E')$ to $Q^k_{B}(E)\times_BB'$ induces $Q^k(e)$ between the fibres and is hence a homotopy equivalence.
\end{rem}

Finally, we also note the following constructions which will be, as in \cite{SST}, useful in our study of products.
If $p:E\to B$ and $p':E'\to B'$ are two fibrations then, there exist, for any $k\geq 0$, commutative diagrams of the following form
 $$\xymatrix{
Q_B^k(E) \times E' \ar[r]\ar[rd]_{q^ k(p)\times p'}& Q_{B\times B'}^{k}(E\times E')\ar[d]^{q^k(p\times p')} &E\times Q_{B'}^k(E')  \ar[r]\ar[rd]_{p\times q^ k(p')}& Q_{B\times B'}^{k}(E\times E')\ar[d]^{q^k(p\times p')}\\
&B\times B' &&B\times B' 
}$$
which are induced by the obvious fibrewise extensions of the maps
$$\begin{array}{rcl}
\Sigma^kZ \times Z' & \to & \Sigma^{k}(Z\times Z')\\
([z,t_1,\ldots,t_k],z') & \mapsto & [(z,z'),t_1,\ldots,t_{k}]\\
\end{array}
$$
and 
$$\begin{array}{rcl}
Z \times \Sigma^kZ' & \to & \Sigma^{k}(Z\times Z')\\
(z,[z',t_1,\ldots,t_k]) & \mapsto & [(z,z'),t_1,\ldots,t_{k}]\\
\end{array}
$$
By considering the fibrewise extensions of the evaluation
$$\begin{array}{rcl}
ev:\Sigma^kQ^k Z & \to & \Sigma^{k}Z\\
\left[\omega,t_1,\ldots,t_k\right] & \mapsto & \omega(t_1,\ldots,t_k)\\
\end{array}
$$
and of the map 
$$\begin{array}{rcl}
Q^k\circ Q^k (Z) & \to & Q^k (Z)\\
\omega:I^k\to \Sigma^kQ^kZ & \mapsto & ev\circ \omega\\
\end{array}
$$
we also get a commutative diagram
$$\xymatrix{
Q^k_B\circ Q^k_B(E) \ar[r]\ar[rd] & Q^k_B(E)\ar[d]\\
&B.
}$$

\noindent This permits us to establish the following result:
\begin{prop}\label{mapproduct} Let $p:E\to B$ and $p':E'\to B'$ be two fibrations. For any $k\geq 0$, there exists a commutative diagram
$$\xymatrix{
Q_B^k(E) \times Q_{B'}^k(E') \ar[rr]\ar[rd]_{q^ k(p)\times q^ k (p')}&  & Q_{B\times B'}^{k}(E\times E')\ar[dl]^{q^k(p\times p')}\\
&B\times B'.
}$$
\end{prop}
\begin{proof} Using the constructions above we obtain:

$$\resizebox{1\textwidth}{!}{\xymatrix{
Q_B^k(E) \times Q_{B'}^k(E') \ar[r]\ar[d]_{q^ k(p)\times q^ k(p')}& Q_{B\times B'}^{k}(E\times Q_{B'}^k(E'))\ar[d]\ar[r]& Q_{B\times B'}^{k}\circ Q_{B\times B'}^{k}(E\times E')\ar[d]\ar[r]& Q_{B\times B'}^{k}(E\times E')\ar[d]^{q^k(p\times p')}\\
B\times B'\ar@{=}[r]&B\times B'\ar@{=}[r]&B\times B'\ar@{=}[r]&B\times B'.
}}$$

\end{proof}

\subsection{Definition of $Q\secat$ and $Q\TC$}\label{DefQsecat}

Let $f:A\to X$ a cofibration where $X$ is a well-pointed path-connected space and let $k\geq 0$. By applying the $Q^k_X$ construction to the Ganea fibrations $g_n(f)$ we define

\begin{defin} $Q^k\secat(f)$ is the least integer $n$ (or $\infty$) such that the fibration $$q^k (g_n(f)):Q_X^k( G_n(f))\to X$$ admits a (homotopy) section.
\end{defin}

\noindent 
By Diagram (\ref{DiagSequenceQk}), we have: 
$$\cdots\leq  Q^ {k+1}\secat(f)\leq Q^k\secat(f) \leq \cdots \leq Q^1\secat(f)\leq Q^0\secat(f)=\secat(f)$$

\noindent As a limit invariant, we set: $Q\secat(f):=\lim Q^k\secat(f)$.\\

\noindent If $f$ is the inclusion $\ast\to X$, we recover the notion of $Q^k\cat(X)$ introduced by H. Scheerer, D. Stanley and D. Tanr\'e. If $X$ is a CW-complex and $f$ is the diagonal map $\Delta:X\to X\times X$, we naturally use the notation $Q^k\TC(X)$ and $Q\TC(X)$.

\begin{rem} The notion of $Q^k\secat$ can be extended to any map $g$ by applying the $Q$ construction to any fibration equivalent to the join map characterizing $\secat(g)$ or, equivalently, by setting $Q^k\secat(g):=Q^k\secat(f)$ where $f$ is any cofibration weakly equivalent to $g$. 
\end{rem}

\section{Some properties of Qsecat and QTC}\label{properties}

In all the statements in this section, we consider cofibrations whose target is a well-pointed path-connected space.

\subsection{Basic properties}

\hspace{2mm}

\noindent We start with the following properties which permit us to generalize to $Q\TC$ and $Q\cat$ the well-known relationships \cite{Farber} between $\TC$ and $\cat$ (see Corollary \ref{firstpropQTC}):

\begin{prop} Let $f:A\to X$ and $g:B\to Y$ be two cofibrations together with a commutative diagram
$$\xymatrix{A \ar[r]_{\varphi} \ar[d]_f & B\ar[d]^g\\
X\ar[r]_{\psi} &Y.
}$$
\begin{enumerate}
\item[(a)] If $\psi$ is a homotopy equivalence then, for any $k\geq 0$, $Q^k\secat(f)\geq Q^k\secat(g)$.
\item[(b)] If the diagram is a homotopy pull-back then, for any $k\geq 0$, we have $Q^k\secat(f)\leq Q^k\secat(g)$.
\end{enumerate}
\end{prop}
\begin{proof}
By considering the Ganea fibrations and applying the $Q^k_{X}$ construction, we obtain the following commutative diagram
$$\xymatrix{Q^k_X(G_n(f)) \ar[r]\ar[d]_{q^k(g_n(f))} & Q^k_Y(G_n(g))\ar[d]^{q^k(g_n(g))}\\
X\ar[r]_{\psi} &Y.
}$$
If $\psi$ is a homotopy equivalence, we deduce from a section of the left-hand fibration a homotopy section of the right-hand fibration, which proves the statement (a). For (b), the hypothesis implies that the diagram is a homotopy pull-back (see Proposition \ref{funct+pullback}) which permits us to obtain a section of the left hand fibration from a section of the right hand one.
\end{proof}

\begin{cor}\label{firstpropQTC} Let $X$ be a path-connected CW-complex. For any $k\geq 0$, we have $Q^k\cat(X)\leq Q^k\TC(X)\leq Q^k\cat(X\times X)$.
\end{cor}

\begin{proof} It suffices to apply the two items of the propostion above to the following two diagrams, respectively. The right hand diagram, where $i_2$ is the inclusion on the second factor, is a homotopy pull-back:

$\qquad \qquad \xymatrix{
 \ast \ar[r]\ar[d] &X\ar[d]^{\Delta} &\quad &\ast\ar[r]\ar[d] & X\ar[d]^{\Delta} \\
 X\times X\ar@{=}[r]& X\times X  &\quad \qquad & X \ar[r]_{i_2}    & X\times X.           
}$

\end{proof}

\subsection{Cohomological lower bound}

Let $f: A\to X$ be a cofibration, and let $f^*\colon H^*(X;{\Bbbk})\to H^ *(A;{\Bbbk})$ be the morphism induced by $f$ in cohomology with coefficients in a field ${\Bbbk}$. As is well-known (\cite{Schwarz}), if we consider the index of nilpotency, ${\rm nil}$, of the ideal $\ker f^*$, that is, the least integer $n$ such that any $(n+1)$-fold cup product in $\ker f^*$ is trivial, then we have
$${\rm nil}\ker f^*\leq \secat(f).$$
Actually the proof of \cite[Thm 5.2]{MTC} permits us to see that 
$${\rm nil}\ker f^*\leq H\secat(f) \leq \secat(f)$$
where $H\secat(f)$ is the least integer $n$ such that the morphism induced in cohomology by the $n$th Ganea fibration of $f$, $H^*(X;{\Bbbk})\to H^*(G_n(f);{\Bbbk})$, is injective.\\

\noindent Here we prove that
\begin{thm}\label{QsecatHsecat} For any $k\geq 0$, ${\rm nil}\ker f^*\leq H\secat(f)\leq Q^k\secat(f)$.
\end{thm}

\begin{proof} Suppose that $Q^k\secat(f)\leq n$. By applying Diagram (\ref{DiagQk}) to the $n$th Ganea fibration of $f$, we obtain the following commutative diagram:
$$\xymatrix{
G_n(f) \ar[d]^{g_n(f)} \ar[r] & Q^k_X(G_n(f)) \ar[d]^{q^k(g_n(f))}\ar[r] & Q^k(G_n(f))\ar[d]^{Q^k(g_n)} \ar[r]^{\sim} &\Omega^k\widetilde{\Sigma}^ k G_n(f)\ar[d]^{\Omega^ k\widetilde{\Sigma}^kg_n(f)}\\
X \ar@{=}[r] & X \ar[r]^{\eta^k}& Q^k(X)\ar[r]^{\sim} &\Omega^k\widetilde{\Sigma}^k X.
}$$
The composite $\tilde{\eta}^ k:X\to Q^k(X)\stackrel{\sim}{\to}\Omega^k\widetilde{\Sigma}^k X$, which is the usual coaugmentation, can be identified to the adjoint of the identity of $\widetilde{\Sigma}^ kX$ through the $k$-fold adjunction.
From $Q^k\secat(f)\leq n$ we know that there exists a map $\psi: X\to \Omega^k\widetilde{\Sigma}^ k G_n(f)$ such that $\Omega^ k\widetilde{\Sigma}^k(g_n(f))\psi \simeq \tilde{\eta}^ k$. Through the $k$ fold adjunction, we get a homotopy section of $\widetilde{\Sigma}^k(g_n(f))$ which implies that $H^ *(g_n(f))$ is injective.
\end{proof}

When $f=\Delta:X\to X\times X$, $\nil\ker\Delta^*$ coincides with Farber's \textit{zero-divisor cup-length}, that is, the nilpotency of the kernel of the cup-product $\cup_X: H^*(X;{\Bbbk})\otimes H^*(X;{\Bbbk})\to H^*(X;{\Bbbk})$ and we have:

\begin{cor}\label{zcl} Let $X$ be a path-connected CW-complex. For any $k\geq 0$, we have $\nil\ker\cup_X\leq Q^k\TC(X)$.\\
\end{cor}

\subsection{Products and fibrations}

We here study the behaviour of $Q\secat$ and $Q\TC$ with respect to products and fibrations and establish generalizations of well-known properties of LS-category and topological complexity. Classically these properties are proven using the open cover definition of these invariants. In order to be able to obtain generalizations to  $Q\secat$, $Q\TC$ and $Q\cat$, we first need a proof of the classical property based on the Ganea fibrations.

\begin{thm}\label{productQsecat} Let $f:A\to X$ and $g:B\to Y$ be two cofibrations. Then, for any $k\geq 0$, $Q^{k}\secat(f\times g)\leq Q^k\secat(f)+Q^k\secat(g)$.

\end{thm}

\begin{proof}
From the diagram constructed in \cite[Pg. 27]{GGV15} (from the product of fibrewise joins of two fibrations to the fibrewise join of the fibration product) we can deduce the existence of a commutative diagram of the following form:
$$\xymatrix{
G_n(f) \times G_m(g)\ar[rr] \ar[rd]_{g_n(f)\times g_m(g)} && G_{m+n}(f\times g) \ar[dl]^{g_{n+m}(f\times g)}\\
& X\times Y.
}$$
By applying the $Q^{k}_{X\times Y}$ construction and using the map of Proposition \ref{mapproduct} we obtain:
$$\resizebox{1\textwidth}{!}{
\xymatrix{
Q^ k_X(G_n(f))\times Q^ k_Y(G_m(g)) \ar[r]\ar[rd]_{q^ k(g_n(f))\times q^ k(g_m(g))}&Q^{k}_{X\times Y}(G_n(f) \times G_m(g))\ar[r] \ar[rd]_{q^{k}(g_n(f)\times g_m(g))} & Q^{k}_{X\times Y}(G_{m+n}(f\times g)) \ar[d]^{q^{k}(g_{n+m}(f\times g))}\\
&X\times Y \ar@{=}[r]& X\times Y.
}}$$
From this diagram, we can establish that, if $Q^k\secat(f)\leq n$ and $Q^k\secat(g)\leq m$ then $Q^{k}\secat(f\times g)\leq n+m$.
\end{proof}

\begin{cor} Let $X$ and $Y$ be path-connected CW-complexes. For any $k\geq 0$, we have $Q^k\TC(X\times Y)\leq Q^k\TC(X)+Q^k\TC(Y)$.
\end{cor}

\begin{proof}
Since the diagonal map $\Delta_{X\times Y}:X\times Y \to X\times Y \times X\times Y$ coincides, up to the homeomorphism switching the two middle factors, with the product $\Delta_X \times \Delta _Y$, we have $Q^ k\TC(X\times Y)=Q^ k\secat(\Delta_X \times \Delta _Y)$ and the result follows from the theorem above.
\end{proof}

We now turn to the study of fibrations. We first establish the following result.

\begin{prop}\label{secatdiagram}
Suppose that there exists a commutative diagram
  $$\xymatrix{
  A\ar[d]_f \ar[r] & C\ar[d]_g\\
  X\ar[r]^{\iota}&Y}$$
 in which $f$, $g$ and $\iota$ are cofibrations. Then, for any $k\geq 0$, we have 
 $$Q^k\secat(g)\leq (Q^k\secat(\iota)+1)(\secat(f)+1)-1.$$
\end{prop}

\begin{proof} Suppose that $\secat(f)\leq p$. Then the fibration $g_p(f)$ admits a section. Using the commutative diagram at the bottom of page 3, we obtain from this section a map $\lambda: X\to G_p(g)$ such that $g_p(g)\circ \lambda=\iota$. Explicitly, for $x\in X$, $\lambda(x)$ is an element $(\gamma_0,\cdots,\gamma_p)\in \Gamma_p(Y)$ such that $\gamma_0(0)=\cdots=\gamma_p(0)=\iota(x)$ and at least one path $\gamma_j$ satisfies $\gamma_j(1)\in \iota f(A)\subset g(C)$. Since $\iota$ is a cofibration and the fibration $\Gamma_p(Y)\to Y$, $(\gamma_0,\cdots,\gamma_p)\mapsto \gamma_0(0)$ is a homotopy equivalence, the relative lifing lemma (\cite[Th. 9]{NoteonCofII}) in the diagram
$$\xymatrix{
X\ar@{ >->}[d]_{\iota}\ar[r]^-{\lambda} & \Gamma_p(Y)\ar@{->>}[d]^{\sim}\\
Y \ar@{=}[r] \ar@{.>}[ru]^{\hat{\lambda}}&Y
}
$$
permits us to extend $\lambda$ to a map $\hat{\lambda}:Y \to \Gamma_p(Y)$ which takes a point $y\in Y$ to an element $(\gamma_0,\cdots,\gamma_p)\in \Gamma_p(Y)$ where all the $\gamma_i$ start in $y$ and at least one path $\gamma_j$ satisfies $\gamma_j(1)\in \iota f(A)$ whenever $y\in \iota(X)$.
The concatenation of a path $\beta_i$ with the path $\hat{\lambda}(\beta_i(1))$ gives, for any $m$, a map
$\lambda_m:G_m(\iota)\longrightarrow G_{pm+p+m}(g)$ such that $g_{pm+p+m}(g)\circ \lambda_m =g_m(\iota)$. As a consequence, we have for any $m\geq 0$ and any $k\geq 0$, a commutative diagram of the following form:
$$\xymatrix{
Q^k_Y(G_m(\iota))\ar[rr]^{Q^k_Y(\lambda_m)} \ar[rd]_{q^ k{(g_m(\iota))}} && Q^k_Y(G_{pm+p+m}(g)) \ar[ld]^{  \,\,\,\,q^ k{(g_{pm+p+m}(g))}}\\
&Y&
}$$
Therefore, if $Q^k\secat(\iota)\leq m$, then there exists a section of the right-hand fibration and we can conclude that $Q^k\secat(g)\leq pm+p+m=(m+1)(p+1)-1$.
\end{proof}

As a consequence we obtain the following property of $Q^k\secat$ (and $\secat$ when $k=0$):

\begin{thm}\label{secatfibration}
Let $\xymatrix{F\ar[r]^{\iota}&E\ar[r]^{\pi}&B}$ be a fibration over a well-pointed space. Suppose that there exists a commutative diagram
  $$\xymatrix{
  A\ar[d]_f \ar[r] & C\ar[d]_g\\
  F\ar[r]^{\iota}&E\ar[r]^{\pi}&B}$$
 in which $f$ and $g$ are cofibrations. Then, for any $k\geq 0$, we have 
 $$Q^k\secat(g)\leq (Q^k\cat(B)+1)(\secat(f)+1)-1.$$
\end{thm}

\begin{proof} Since $B$ is well-pointed, $\iota$ is a cofibration (\cite[Th. 12]{NoteonCofII}). By Proposition \ref{secatdiagram}, we have $Q^k\secat(g)\leq (Q^k\secat(\iota)+1)(\secat(f)+1)-1$. On the other hand, 
since the following diagram
$$\xymatrix{
Q^k_{E}(G_m(\iota)) \ar[r]\ar[d]_{q^ k{(g_m(\iota))}} & Q^k_{B}(G_n(B)) \ar[d]^{q^ k{(g_m(B))}}\\
E \ar[r]_{\pi} &B
}$$
is a homotopy pull-back, we have $Q^k\secat(\iota)\leq Q^k\cat(B)$ and the result follows.
\end{proof}

\begin{cor}\label{Corfibration}
Let 
$\xymatrix{F\ar[r]^{\iota}&E\ar[r]^{\pi}&B}$ 
be a fibration where all the spaces are well-pointed path-connected CW-complexes. For any $k\geq 0$, we have 
\begin{enumerate}
\item[(a)] $Q^k\cat(E)\leq (Q^k\cat(B)+1)(\cat(F)+1)-1$.
\item[(b)] $Q^k\TC(E)\leq (Q^k\cat(B\times B)+1)(\TC(F)+1)-1$.
\end{enumerate}
\end{cor}

\begin{proof} We apply the theorem above to the following diagrams, respectively:
$$\xymatrix{
  \ast\ar[d] \ar[r] & \ast\ar[d]\\
  F\ar[r]^{\iota}&E\ar[r]^{\pi}&B}
 \quad \quad  \xymatrix{
  F\ar[d]_{\Delta} \ar[r]_{\iota} & E\ar[d]_{\Delta}\\
  F\times F\ar[r]^{\iota\times \iota}&E\times E\ar[r]^{\pi\times \pi}&B\times B.}$$
\end{proof}

\begin{rem} If we consider Corollary \ref{Corfibration} for $k=0$, the formula given in $(a)$ is the classical formula for LS-category while the formula given in $(b)$ corresponds to the formula given in \cite{FarberGrant}. It seems that these two special cases are the only special cases of the formulas established in  Proposition \ref{secatdiagram} and Theorem \ref{secatfibration} which were already known. In particular, Corollary \ref{Corfibration}(a) was not known for $k\geq 1$ and we did not find in the literature formulas corresponding to the case $k=0$ of Theorem \ref{secatfibration}. 
\end{rem}

We finally also establish a generalization to $Q\cat$ of a property of the LS-category due to O. Cornea \cite{Cornea}. We will next use this result in our final remark on the Scheerer-Stanley-Tanr\'e conjecture.

\begin{thm}\label{Cornea} Let $\xymatrix{F\ar[r]^{\iota}&E\ar[r]^{\pi}&B}$ be a fibration where all the spaces are well-pointed path-connected  CW-complexes and let $k\geq 0$. If $Q^ k\cat(B)\geq 1$ or $B$ is simply-connected, then $Q^k\cat(E/F)\leq Q^ k\cat(B)$.
\end{thm} 

\begin{proof} Consider the following commutative diagram where all the vertical maps are cofibrations, the bold square is a homotopy pull-back, $\psi$ is the identification map and $\rho$ is the map induced by $\pi$:
$$\resizebox*{4cm}{3cm}{
\xymatrix{
F \ar@{=}[r] \ar@{=}[dd]& F\ar[rr] \ar@<0.2mm>[rr] \ar[rd] \ar[dd]_{\iota}\ar@<0.2mm>[dd]&& \ast \ar[dd]\ar@<0.2mm>[dd]\\
&& \ast \ar@{=}[ru]\ar[dd]\\
F \ar[r]_{\iota} & E\ar[rr] \ar@<0.2mm>[rr] \ar[rd]_{\psi} && B \\
&& E/F \ar[ru]_{\rho}\\
}
}$$
By taking the Ganea fibrations, we obtain the following commutative diagram where the bold square is a homotopy pull-back:
$$\resizebox*{7cm}{4cm}{
\xymatrix{
F \ar[r]^{\phi} \ar@{=}[dd]& G_n(\iota)\ar[rr] \ar@<0.2mm>[rr] \ar[rd] \ar[dd]_{g_n(\iota)}\ar@<0.2mm>[dd]&& G_n(B) \ar[dd]^{g_n(B)}\ar@<0.2mm>[dd]\\
&& G_n(E/F) \ar@{=}[ru]\ar[dd]^<<<<<{g_n(E/F)}\\
F \ar[r]_{\iota} & E\ar[rr] \ar@<0.2mm>[rr] \ar[rd]_{\psi} && B \\
&& E/F \ar[ru]_{\rho}\\
}
}$$
The map $\phi$ takes a point $f\in F$ to the constant $(n+1)$-tuple of paths $(\hat{f},\ldots,\hat{f})\in G_n(\iota)$. It is clear that the composite of $\phi$ with the map $G_n(\iota)\to G_n(E/F)$ is trivial. We now apply the $Q^k_{-}$ construction to obtain the following diagram
$$\resizebox*{7cm}{4cm}{
\xymatrix{
F \ar[r]^{\hat{\phi}} \ar@{=}[dd]& Q^k_E(G_n(\iota))\ar[rr] \ar@<0.2mm>[rr] \ar[rd] \ar[dd]_{q^k(g_n(\iota))}\ar@<0.2mm>[dd]&& Q^k_B(G_n(B)) \ar[dd]^{q^k(g_n(B))}\ar@<0.2mm>[dd]\\
&& Q^k_{E/F}(G_n(E/F)) \ar@{=}[ru]\ar[dd]^<<<<<{q^k(g_n(E/F))}\\
F \ar[r]_{\iota} & E\ar[rr] \ar@<0.2mm>[rr] \ar[rd]_{\psi} && B \\
&& E/F \ar[ru]_{\rho}\\
}}
$$
where $\hat{\phi}(f)$ is the map $I^k\to \Sigma^k_EG_n(\iota)$ given by $u\mapsto [(\hat{f},\ldots,\hat{f}),u]$. Once again the bold square is a homotopy pull-back and the composite of $\hat{\phi}$ with the map $Q^k_EG_n(\iota)\to Q^k_{E/F}G_n(E/F)$ is trivial (since its sends $f$ to the base point of the space $Q^k_{E/F}(G_n(E/F))$). Suppose that $Q^ k\cat(B)\leq n$. Then the fibration $q^k(g_n(B))$ admits a section $s$ that we can suppose to be pointed (because, under our hypothesis $n\geq 1$ or $B$ simply-connected, the space $Q_B^k(G_n(B))$ is path-connected). By the unicity (up to homotopy) of the whisker map in the bold homotopy pull-back, we then obtain a homotopy section $s'$ of $q^k(g_n(\iota))$ which satisfies $s'\circ \iota\simeq \hat{\phi}$. Therefore the composition of $s'\circ \iota$ with the map $Q^k_E(G_n(\iota))\to Q^k_{E/F}(G_n(E/F))$ is homotopically trivial and $s'$ induces a map $\tilde{s}:E/F \to Q^k_{E/F}(G_n(E/F))$ giving the following diagram:
$$
\xymatrix{
F  \ar@{=}[d]\ar[r]_{\iota} & E\ar[d]_{s'} \ar[rr]_{\psi} && E/F \ar[d]_{\tilde{s}}\\
F \ar[r]^{\hat{\phi}} \ar@{=}[d]& Q^k_E(G_n(\iota))\ar[rr] \ar[rr] \ar[d]_{q^k(g_n(\iota))}\ar[d]&&  Q^k_{E/F}(G_n(E/F))\ar[d]_{q^k(g_n(E/F))} \\
F \ar[r]_{\iota} & E\ar[rr]_{\psi} && E/F \\
}
$$
Since $F\to E\to E/F$ is a cofibre sequence and the left and middle vertical composites are homotopy equivalences we obtain that $q^k(g_n(E/F))\circ\tilde{s}$ is a homotopy equivalence. We can thus conclude that $q^k(g_n(E/F))$ admits a homotopy section and that $Q^k\cat(E/F)\leq n$. 
\end{proof}

\subsection{Condition for $Q\secat=\secat$}

Similarly to \cite[Theorem 15]{Qcat} we have 
\vspace{0.1cm}
\begin{thm} \label{dimconnectivity} Let $f:A\to X$ be a cofibration which is an $(r-1)$-equivalence with $r\secat(f) \geq 3$. If
$\dim (X) \leq {2r \secat(f) -3}$ then $Q\secat(f)=\secat(f)$.
\end{thm}

\noindent This result follows by induction from the following proposition which is a generalization of \cite[Prop. 16]{Qcat}:

\begin{prop}
\label{prop-dim-connexite}
Let $f:A\to X$ be a cofibration which is an $(r-1)$-equivalence and $n$ be an integer such that $rn\geq 3$. If for $k\geq 0$
one has $Q^k\secat(f)= n$ and $\dim (X) \leq 2rn
 -3+k$ then $Q^{k+1}\secat(f)=n$.
\end{prop}

\begin{proof}
Since $Q^{k+1}\secat(f)\leq Q^k\secat(f)= n$, we only have to prove that
$Q^{k+1}\secat(f)\geq n$. Suppose that $Q^{k+1}\secat(f)\leq n-1$ and let $\sigma : X \to
{Q_X^{k+1}}(G_{n-1}(f))$ be a (homotopy) section of
$q^{k+1}(g_{n-1}(f))$. Consider the
following diagram:
$$\xymatrix{
Q^{k}(F_{n-1}(f))\ar[rr]^{b^k} \ar[d]&&  Q^{k+1}(F_{n-1}(f)) \ar[d]\\
Q^{k}_X(G_{n-1}(f))\ar[rr]^{b_X^k} \ar[rd]_{q^ k(g_{n-1}(f))} &&
{Q^{k+1}_X}(G_{n-1}(f)) \ar[ld]^{q^{k+1}(g_{n-1}(f))}\\
&X&
}$$
Since $f$ is an $(r-1)$-equivalence, the fibre of $f$ is $(r-2)$-connected and the space $F_{n-1}(f)$ is $(rn-2)$-connected. The condition $rn\geq 3$
ensures that $F_{n-1}(f)$ is at least 1-connected. Using the fact that, for a $(l-1)$-connected space $Z$ with $l\geq 2$, the coaugmentation $\tilde{\eta}_Z:Z\to \Omega \widetilde{\Sigma} Z$ is a $(2l-1)$-equivalence, we can see that the
 map $b^k: Q^{k}(F_{n-1}(f))\to Q^{k+1}(F_{n-1}(f))$, which is equivalent to $\Omega^k\tilde{\eta}_{\widetilde{\Sigma}^kF_{n-1}(f)}$, is a
$(2rn-3+k)$-equivalence for all $k\geq 0$. As a consequence the map $b_X^k: {Q^{k}_X}(G_{n-1}(f))\to
{Q^{k+1}_X}(G_{n-1}(f))$ is also a $(2rn-3+k)$-equivalence for all $k\geq 0$. Since $\dim (X) \leq 2rn
+k -3$, we obtain a map $\sigma' : X \to
{Q^{k}_X}(G_{n-1}(f))$ such that $b^k_X \circ \sigma' \simeq
\sigma$. This map is a homotopy section of
$q^k(g_{n-1}(f))$ which contradicts the hypothesis $Q^k\secat(f)=
n$. Therefore $Q^{k+1}\secat(f)\geq n$.
\end{proof}

\vspace{0.1cm}
\begin{cor} \label{dimconnectivityTC} Let $X$ be an $(r-1)$-connected CW-complex such that \linebreak $r\TC(X) \geq 3$. If
$\dim (X) \leq \disfrac{2r \TC(X) -3}{2}$ then $Q\TC(X)=\TC(X)$.
\end{cor}

\begin{proof} If $X$ is $(r-1)$-connected then the diagonal map $\Delta: X\to X\times X$ is an $(r-1)$ equivalence and we can apply Theorem \ref{dimconnectivity}.
\end{proof}

\subsection{Examples and observations}


\begin{enumerate}
\item Let $Y=\R P^6/\R P^2$ and $X=Y\vee Y$. In  \cite{GGV15}, it is proved that $\TC(X)=\nil\ker\cup_X=4$ and that $X$ is a counter-example for the analogue of the Ganea conjecture for $\TC$ since, for any even $n$, $\TC(X\times S^{n})=5<\TC(X)+\TC(S^{n})$. By Theorem \ref{QsecatHsecat}, it is clear that a space $Z$ satisfying $\nil\ker \cup_Z=\TC(Z)$ also satisfies $Q\TC(Z)=\TC(Z)$. Therefore the space $X=Y\vee Y$ satisfies $Q\TC(X)=\TC(X)=4$. The interest of this example is that it shows that the analogue for $\TC$ of the Scheerer-Stanley-Tanr\'e conjecture mentioned in the introduction (which would be $\TC(X\times S^n)=\TC(X)+\TC(S^n)$ if and only if $Q\TC(X)=\TC(X)$) does not hold in general. Actually the analogue of \cite[Thm 11]{Qcat} for topological complexity is not true: we cannot have  $Q\TC(X\times S^n)=Q\TC(X)+Q\TC(S^n)$ because, in that case, it would be true that $Q\TC(X)=\TC(X)$ implies that $\TC(X\times S^n)=\TC(X)+\TC(S^n)$. Note that $Q\TC(S^n)=\TC(S^n)$ since $\nil \ker \cup_{S^n}=\TC(S^n)$.\\

\item Let $X=S^3\cup_{\alpha} e^7$ where $\alpha:S^6\to S^3$ is the Blakers-Massey element. We know by \cite{GC-V} (or \cite{GGV15}) that $\TC(X)=3$. Therefore the condition $\dim (X) \leq \disfrac{2r \TC(X) -3}{2}$ (with $r=3$) is satisfied and $Q\TC(X)=3$ by Corollary \ref{dimconnectivityTC}. This is an example for which $\nil\ker\cup_X < Q\TC(X)$. Actually, using \cite[Theorem 5.5]{GGV15}, we know that for any two-cell complex $X=S^p\cup_{\alpha} e^{q+1}$ such that 
\begin{enumerate}
\item $2p-1<q\leq 3p-3$ 
\item the Hopf invariant of $\alpha$, $H_0(\alpha)\in \pi_q(S^{2p-1})$, satisfies the condition $(2+(-1)^p)H_0(\alpha)\neq 0$ 
\end{enumerate}
we have $\TC(X)\geq 3$. Corollary \ref{dimconnectivityTC} permits us to see that for these spaces we have $Q\TC(X)=\TC(X)$.\\

\item Let $X=S^2\cup_{\alpha} e^{10}$ be the space considered by Iwase in \cite{Iwase} to disprove the Ganea conjecture. It is easy to see that $\nil\ker\cup_X=2$ and, using Hopf invariants, it has been proved in \cite{GGV15} that $\TC(X)\leq 3$. Here we can compute that $Q\TC(X)=2$. Indeed, by Corollaries \ref{firstpropQTC} and \ref{zcl}, we have
$$\nil\ker\cup_X\leq Q\TC(X)\leq Q\cat(X\times X).$$
On the other hand, we know from \cite{SST} (or Theorem \ref{productQsecat}) that $Q\cat(X\times X)\leq 2Q\cat(X)$. We also know that, for this particular space $X=S^2\cup_{\alpha} e^{10}$, we have $Q\cat(X)=1$ (\cite[Ex. 5]{SST}). We therefore can conclude that $Q\TC(X)=2$. More generally, any finite CW-complex $X$ which is a counter-example to the Ganea conjecture satisfies $Q\cat(X)<\cat(X)$ (see below) and therefore $Q\TC(X)\leq 2\cat(X)-2$.\\

\item As mentioned in \cite[Remark 7]{MoyauxV}, using \cite[Theorem 15]{Qcat} (or Theorem \ref{dimconnectivity} above for the cofibration $\ast \to X$),
we can see that, for any path-connected finite dimensional CW-complex $X$, the product $X\times T^p$, where $T^p$ is the $p$-fold product of $p$ circles $S^1$, satisfies $Q\cat(X\times T^p)=\cat(X\times T^p)$ as soon as $p\geq \dim(X)+3$. In a similar way, considering $p$-fold products of the sphere $S^2$, we have:
\vspace{0.3cm}

\begin{itemize}\item If $X$ is a $1$-connected CW-complex then, for any $p\geq \disfrac{2\dim(X)+3}{4}$, the space  
$X\times (S^2)^p$ satisfies $Q\TC(X\times(S^2)^p)= \TC(X\times(S^2)^p)$.
\end{itemize}

\vspace{0.3cm}
\noindent Indeed the space $X\times (S^2)^p$ is $1$-connected so that we can consider $r=2$ in Corollary \ref{dimconnectivityTC}. On the other hand we have $\dim(X\times (S^2)^p)=\dim(X)+2p$ and $\TC(X\times (S^2)^p) \geq 2p$ since $X\times (S^2)^ p$ dominates $(S^2)^ p$ whose topological complexity is $2p$. Then if  $p\geq \disfrac{2\dim(X)+3}{4}$ we get $2(\dim(X)+2p)\leq 8p-3$ and Corollary \ref{dimconnectivityTC} permits us to conclude.\\

\end{enumerate}

\noindent We finish this text with a small remark on the Scheerer-Stanley-Tanr\'e conjecture:\\

\noindent
\textbf{Conjecture. }(\cite{SST}) Let $X$ a finite CW-complex. We have
$\cat(X\times S^ n)=\cat(X)+1$ for all $n\geq 1$ ($X$ satisfies the Ganea conjecture) if and only if $Q\cat(X)=\cat(X)$.\\

As mentioned in the introduction, one direction of the Scheerer-Stanley-Tanr\'e conjecture has been proved in \cite{Qcat}. More precisely, using the fact that the invariant $Q\cat$ satisfies $Q\cat(X\times S^ n)=Q\cat(X) +1$ (\cite[Corollary 12]{Qcat}), we have that $\cat(X\times S^ n)=\cat(X)+1$ holds for all $n\geq 1$ as soon as $Q\cat(X)=\cat(X)$.

In the last item above we recalled that for any path-connected finite dimensional CW-complex $X$, the product $X\times T^p$, where $T^p$ is the $p$-fold product of $p$ circles $S^1$, satisfies $Q\cat(X\times T^p)=\cat(X\times T^p)$ and therefore the Ganea conjecture as soon as $p\geq \dim(X)+3$. We also notice the following consequence of \cite[Corollary 12]{Qcat}:

\begin{itemize}\item If we have $Q\cat(X\times T^p)=\cat(X\times T^p)$ for some $p$ then we have $Q\cat(X\times T^q)=\cat(X\times T^q)$ for all $q\geq p$.
\end{itemize}

\noindent Indeed, suppose that $Q\cat(X\times T^q)=\cat(X\times T^q)$. Then $X\times T^q$ satisfies the Ganea conjecture and, since $X\times T^{q+1}=X\times T^q\times S^1$, we have
$$\cat(X\times T^{q+1})=\cat(X\times T^q)+1=Q\cat(X\times T^q)+1=Q\cat(X\times T^{q+1}).$$
The assertion follows by induction.\\

We now use Theorem \ref{Cornea} to ask a question related to the Scheerer-Stanley-Tanr\'e conjecture:\\

Suppose that $X$ is a counter-example to the Ganea Conjecture and let $p$ be the least integer such that $X\times T^p$ satisfies the Ganea Conjecture. Consider the space $Y=(X\times T^p)/S^1$ and the following diagram where the dotted map is induced by the projection $X\times T^p\to X\times T^{p-1}$:
$$\xymatrix{
S^1\ar[r] & X\times T^p \ar[d]\ar[r]^{{\rm pr}} & X\times T^{p-1}\\
& Y=(X\times T^p)/S^1 \ar@{.>}[ru]
}$$
Since $X\times T^{p-1}$ does not satisfy the Ganea conjecture, we have, by Theorem \ref{Cornea}, 
$$Q\cat(Y)\leq Q\cat(X\times T^{p-1})<\cat(X\times T^{p-1}).$$
On the other hand, since $Y$ dominates $X\times T^{p-1}$ we also have $\cat(Y)\geq \cat(X\times T^{p-1})$. Therefore $Q\cat(Y)<\cat(Y)$.\\

\noindent \textit{Question. } Does $Y$ satisfy the Ganea conjecture?

\noindent Luc\'{\i}a Fern\'andez Su\'arez, \'Area Departamental de Matem\'atica
Instituto Superior de Engenharia, Rua Conselheiro Em\'{\i}dio Navarro, 1, 1959 - 007 Lisboa
Portugal.
{\it{E-mail:}} lsuarez@adm.isel.pt\\

\noindent Lucile Vandembroucq,
Universidade do Minho,
Centro de Matem\'atica,
Campus de Gualtar,
4710-057 Braga,
Portugal.
{\it{E-mail:}} lucile@math.uminho.pt\\

\end{document}